\documentclass[12pt]{article}

\usepackage[margin=1in]{geometry}

\usepackage{amsmath,amssymb,amsthm,mathtools}

\usepackage{booktabs}
\usepackage{tabularx}
\usepackage{longtable}
\usepackage{enumerate}
\usepackage{lineno}

\usepackage{graphicx}
\usepackage{tikz}
\usetikzlibrary{arrows.meta,calc,decorations.pathreplacing,positioning}

\usepackage{pdflscape}

\usepackage{microtype}

\usepackage{fvextra}
\DefineVerbatimEnvironment{NumWrap}{Verbatim}{breaklines,breakanywhere,fontsize=\footnotesize}

\usepackage[hidelinks]{hyperref}

\newtheorem{remark}{Remark}
\newtheorem{theorem}{Theorem}
\newtheorem{lemma}{Lemma}
\newtheorem{proposition}{Proposition}
\newtheorem{corollary}{Corollary}
\numberwithin{equation}{section}
\newtheorem{definition}{Definition}

\title{High-Precision Framework for Expected Hitting Times Analysis in the Dice-Sum Process}

\author{
Tipaluck Krityakierne\thanks{Department of Mathematics, Faculty of Science, Mahidol University, Bangkok 10400, Thailand; Centre of Excellence in Mathematics, MHESI, Bangkok 10400, Thailand. Email: \texttt{tipaluck.kri@mahidol.ac.th}}
\and
Thotsaporn Aek Thanatipanonda\thanks{Science Division, Mahidol University International College, Nakhon Pathom 73170, Thailand. Email: \texttt{thotsaporn@gmail.com}}
}

\date{}

\begin{document}

\maketitle

\begin{abstract}
We study the expected number of rolls required for the cumulative sum of a fair six-sided die to first enter a prescribed target set~$\mathcal{H}\subset\mathbb{Z}_{\ge0}$.
A one-variable dynamic-programming formulation is introduced that removes dependence on the roll count.
Within this framework, the infinite process is truncated at a large cutoff~$N$ and corrected by an analytically derived overshoot term that accounts for the rare event of exceeding~$N$ before entering~$\mathcal{H}$.
Explicit bounds on this residual yield a strict two-sided estimate of the truncation error.
The method is numerically efficient, requiring constant memory and linear time in the cutoff.
For the perfect-square target set $\mathcal{H}=\{n^2:n\in\mathbb{N}\}$, all quantities are evaluated explicitly, yielding
\[
\mathbb{E}[T]=7.07976423755110510389555305690818489468\ldots,
\]
provably correct to 1{,}017 decimal places.
This constitutes the most precise result known to date and establishes a general framework for high-accuracy computation of discrete hitting times.
\end{abstract}

\noindent\textbf{Keywords:}
Expected hitting time; truncated dynamic programming; dice-sum process; perfect squares; high-precision computation.

\medskip

\noindent\textbf{2020 Mathematics Subject Classification:}
60J10; 65C40; 68W25; 60-08.

\bigskip

\section{Introduction}

Consider repeatedly rolling a fair six-sided die and accumulating the sum of outcomes until the total belongs to a prescribed target set~$\mathcal H \subset \mathbb{Z}_{\ge0}$.  
We are interested in the expected number of rolls required for this \emph{hitting time}—the first time the cumulative sum enters~$\mathcal H$.  
Formally, let $S_t$ denote the cumulative sum after $t$ independent die rolls (with $S_0=0$ for the moment), and define the stopping time
\[
T=\inf\{t\ge0: S_t \in \mathcal H\}.
\]
Our objective is to compute the expected value $\mathbb{E}[T]$ with \emph{rigorous high precision}.  A canonical example is $\mathcal H=\{n^2:n\in\mathbb N\}$, corresponding to the case in which the cumulative sum is required to reach a perfect square.  
In this example, $\mathcal H$ is \emph{sparse}, with distances between successive elements increasing without bound, so the hitting-time behavior differs markedly from that of classical problems involving a single threshold or dense target sets.

Hitting times for random walks and renewal processes are fundamental objects of study, dating back to classical analyses such as gambler’s ruin and renewal theory~\cite{Feller1957,Karlin1968}.  
Here, $(S_t)$ forms a renewal process with \textcolor{black}{independent and identically distributed (i.i.d.)~increments}, each uniformly distributed on $\{1,\dots,6\}$, whose first entrance into~$\mathcal H$ defines a stopping event of renewal type~\cite{Cox1962}.
For dense or regularly spaced target sets, renewal theory provides explicit formulas and asymptotic estimates, whereas for sparse or irregular sets~$\mathcal H$, such as the perfect squares, prime numbers, or recurrence sequences, classical first-passage results are no longer directly applicable.
Since the increment distribution is aperiodic, that is, the possible step sizes
$\{1,\dots,6\}$ have greatest common divisor $1$, the process can reach every sufficiently large integer with positive probability, and many natural infinite target sets are hit with probability~$1$~\cite{Feller1971}.
The main challenge lies in computing the expected hitting time accurately.

To address this challenge, we propose a dynamic-programming framework that represents the expected hitting time by a \emph{one-variable backward recursion} on an unbounded state space.  
While conceptually related to truncation-based methods in reinforcement learning and control theory, such as the Truncated Approximate Dynamic Programming (TADP) framework~\cite{Farahmand2016} for Markov Decision Processes with large but finite horizons, our approach differs fundamentally.  
Their formulation concerns a finite-horizon problem and employs a learned, task-dependent terminal value, whereas ours treats an \emph{infinite-horizon} process with a fixed analytic boundary and derives rigorous probabilistic bounds on the truncation error.

In our framework, the infinite process is truncated at a large cutoff~$N$ and corrected by an analytically derived \emph{overshoot term} that accounts for the rare event of exceeding~$N$ before entering~$\mathcal{H}$.  
Explicit upper and lower bounds for this residual yield a strict two-sided estimate of the truncation error.  
This combination of analytic error control and a memory-efficient one-variable recursion provides certified high-precision estimates of~$\mathbb{E}[T]$, applicable to a broad class of target sets.

\paragraph{Historical development of dice-sum problems}
In 2017, DasGupta~\cite{DG} discussed the expected number of die rolls $\mathbb{E}[\tau]$ required for the cumulative sum to reach a \emph{prime number}, in a puzzle published in the \emph{Bulletin of the Institute of Mathematical Statistics}. 
He gave a lower bound of approximately~2.34 for $\mathbb{E}[\tau]$ but did not provide a rigorous exact value. 
The problem subsequently appeared in Matthew~M.~Conroy’s 2018 online collection \emph{A Collection of Dice Problems}~\cite{Con}, where Problem~41 revisited the prime-hitting case and presented an empirical estimate of $\mathbb{E}[\tau]$ to 500~decimal places, accompanied by the cautious remark that no proof of correctness was known. 
Conroy also posed the perfect-square analogue (Problem~40), obtaining
\[
\mathbb{E}[T] = 7.079764237551105103895,
\]
which was rigorously verified to be correct to 21~decimal places.
His computation simultaneously tracked two variables, namely the roll count and the cumulative sum, requiring a two-dimensional iteration.

Recently, a rigorous analysis for the case $\mathcal{H}$ being the set of prime numbers was carried out by Alon and Malinovsky~\cite{AM}, who developed a dynamic-programming formulation defining $p(r,n)$, similar to the two-variable method of Conroy, as the probability that the cumulative sum equals a non-prime~$n$ after $r$ rolls without having previously hit a prime. Their computed value of $\mathbb{E}[\tau]$ was proved to be correct to 7~decimal places.
Related studies by Martinez and Zeilberger~\cite{Z1}, as well as the subsequent collaboration of Alon, Malinovsky, Martinez, and Zeilberger~\cite{Z2}, analyzed analogous dice-sum hitting problems for prime and other arithmetic target sets, using a symbolic generating-function approach. These formulations track both the roll count and cumulative sum, resulting in a two-dimensional state space with greater computational and storage requirements.

\paragraph{Contributions and novelty}
Building on these developments, we introduce a computationally efficient framework for evaluating expected hitting times to a target set in discrete random walks.
\begin{enumerate}[(i)]
\item In contrast to earlier two-variable formulations, our method expresses the dice-sum hitting problem as a \emph{one-variable backward recurrence} that removes roll-count dependence and uses only the six most recent values, yielding \emph{constant-memory} complexity.
\item The framework provides rigorous analytical error bounds and certified high-precision estimates, offering explicit two-sided control  of the truncation error and scaling efficiently to very large cutoffs~$N.$ 
\item For the specific target set $\mathcal H = \{n^2 : n \in \mathbb N\}$, the correction constants are determined rigorously, yielding results accurate to 1{,}017~decimal places and exceeding the 21-decimal-place precision previously reported in~\cite{Con}.
\end{enumerate}

Beyond this example, the approach applies broadly to other target sets, such as primes or polygonal numbers.
Together, these advances establish a rigorous and efficient connection between discrete probability, numerical analysis, and computational simulation.


The remainder of this paper is organized as follows.  
Section~\ref{sec:truncated} presents the general dynamic-programming framework for arbitrary target sets~$\mathcal{H}$ and establishes its well-posedness, monotonicity, and rigorous two-sided truncation bounds (Theorem~\ref{thm:truncation_error_strict}).  
Section~\ref{sec:randomrolls} examines the hit–skip dynamics of the cumulative-sum process, deriving exponentially decaying bounds for the probability of ever hitting a given integer (Lemma~\ref{lemma:bound_p}).  
Section~\ref{sec:squares} applies the framework to the perfect-square target set and presents high-precision numerical results (Theorem~\ref{thm:expected-square}).  
Finally, Section~\ref{sec:future} discusses extensions to other classes of target sets and outlines directions for future work.

\section{Truncated Backward Recursion for the Expected Hitting Time}
\label{sec:truncated}

In this section we formulate a  dynamic-programming framework for the 
expected time required for a random walk to hit a prescribed set of target states.  
All results are stated for an arbitrary target set 
$\mathcal H\subset\mathbb Z_{\ge0}$, which may represent any subset of the nonnegative 
integers.  

\medskip
Fix a starting point $s \in \mathbb{Z}_{\ge 0}$, and define the cumulative-sum process
\begin{equation}
\label{eq:S}
S_t := s + \sum_{i=1}^t X_i, 
\qquad S_0 := s,
\end{equation}
where $(X_i)_{i \ge 1}$ are i.i.d.\ random variables uniformly distributed on $\{1,\dots,6\}$,
representing successive outcomes of fair die rolls.

Define the stopping time
\[
T := \inf\{\, t \ge 0 : S_t \in \mathcal H \,\},
\]
that is, the first time the cumulative sum enters the target set~$\mathcal H$.

For each starting point $s\ge0$, define the expected number of rolls until the cumulative sum first hits~$\mathcal H$ by
\[
E(s):=\mathbb{E}_s[T],
\]
where $\mathbb{E}_s[\cdot]$ denotes expectation conditional on $\{S_0=s\}$.
When $S_0 \in \mathcal{H}$, the definition of $T$ implies $T=0$, so the process is immediately absorbed and $E(s)=0$ for such $s$.

Conditioning on the outcome of the first die roll gives
\[
T =
\begin{cases}
0, & s\in\mathcal{H},\\[4pt]
1 + T', & s\notin\mathcal{H},
\end{cases}
\]
where $T'$ denotes the remaining time to reach~$\mathcal{H}$ starting from the new state
$S_1 = s + X_1$.

Taking expectations conditional on $\{S_0 = s\}$, we obtain the backward recursion
\begin{equation}\label{eq:main}
E(s) =
\begin{cases}
0, & s \in \mathcal{H}, \\[6pt]
1 + \tfrac{1}{6}\displaystyle\sum_{i=1}^6 E(s+i), & s \notin \mathcal{H},
\end{cases}
\end{equation}
for all $s \ge 0$.

Since \textcolor{black}{recursion}~\eqref{eq:main} involves infinitely many states, direct computation is not feasible.
A standard approximation is to truncate the state space by restricting \textcolor{black}{the running sum to values up to a large cutoff value~$N \in \mathbb{N}$}.
Formally, for a fixed cutoff~$N \in \mathbb{N}$, the truncated expectation is defined by
\begin{equation}\label{eq:ENdef}
E_N(s) \;=\;
\begin{cases}
0, & s \in \mathcal H, \ s  \le N, \\[6pt]
1 + \tfrac{1}{6}\displaystyle\sum_{i=1}^6 E_N(s+i), & 0 \le s \le N, \ s\notin\mathcal H,\\[10pt]
0, & s > N.
\end{cases}
\end{equation}
The last line imposes a zero terminal boundary beyond~$N$, closing the backward recursion on~${0,\dots,N}$, which is then solved recursively from~$N$ down to~$s$.

To formalize the truncated framework, we collect and define the relevant stopping times and the associated overshoot event in the following definition, which will be used throughout this work.

\begin{definition}[Stopping times and overshoot event]
\label{def:stopping-times}
Let $\mathcal{H} \subset \mathbb{Z}_{\ge 0}$ be a target set and $N \in \mathbb{N}$ a cutoff.  
Define
\[
T := \inf\{\, t \ge 0 : S_t \in \mathcal{H} \,\}, \qquad
\omega_N := \inf\{\, t \ge 0 : S_t > N \,\}, \qquad
A_N := \{\, \omega_N < T \,\},
\]
where $T$ is the first hitting time of the target set~$\mathcal{H}$,
$\omega_N$ is the first time the cumulative sum exceeds the cutoff~$N$,
and $A_N$ is the \emph{overshoot event}, corresponding to paths that cross~$N$
before reaching~$\mathcal{H}$.
\end{definition}

\begin{proposition}[Probabilistic representation and error decomposition]
\label{prop:identification}
For every \(N \in \mathbb{N}\) and \(s \ge 0\), define
\[
F(s) := \mathbb{E}_s\big[\min(T,\omega_N)\big].
\]
Then, \(F\) satisfies the same recursion and boundary conditions as~\eqref{eq:ENdef}, and hence provides a solution of the truncated recursion.
Moreover,
\begin{equation}\label{eq:err-identity-F}
E(s) - F(s)
\;=\;
\mathbb{E}_s\!\big[(T - \omega_N)\,\mathbf{1}_{\{\omega_N < T\}}\big].
\end{equation}
In words, \(\min(T,\omega_N)\) represents the truncated stopping time, and
\((T-\omega_N)\mathbf{1}_{\{\omega_N<T\}}\) is the additional waiting time incurred along paths that overshoot the cutoff \(N\) before hitting the target set~\(\mathcal{H}\).
\end{proposition}

\begin{proof}
A first-step analysis gives, for non-absorbing $s\le N$ that are not in $\mathcal H$,
\[
F(s)=1+\tfrac16\sum_{i=1}^6 F(s+i).
\]
If $s\in\mathcal H,$ then $\min(T,\omega_N)=0$, so $F(s)=0$. 
If $s>N$ then $\omega_N=0$, hence $F(s)=0$. 
Thus, $F$ satisfies the same system \eqref{eq:ENdef}, showing that it is a 
solution of the truncated recursion. 

For the error identity \eqref{eq:err-identity-F}, decompose
\[
T=\min(T,\omega_N)+\bigl(T-\omega_N\bigr)\,\mathbf{1}_{\{\omega_N<T\}},
\]
take $\mathbb{E}_s[\cdot]$, and use $\mathbb{E}_s[\min(T,\omega_N)]=F(s)$.
\end{proof}

Having established a probabilistic representation $F$ that solves \eqref{eq:ENdef}, 
we now verify that the truncated recursion \eqref{eq:ENdef} 
is well posed and that its solution increases with the cutoff parameter~$N$.

\begin{proposition}[Well-posedness and monotonicity]\label{prop:mono}
For each $N\in\mathbb{N}$ the system \eqref{eq:ENdef} has a unique solution 
$E_N:\mathbb{Z}_{\ge0}\to[0,\infty)$. Moreover, for any $N'<N$ and any $s\ge0$,
\[
E_{N'}(s)\le E_N(s)\le E(s).
\]
\end{proposition}

\begin{proof}
\emph{Existence and uniqueness.}
By backward dynamic programming, the values at states $s \le N$ are determined 
recursively from those at $s+1,\dots,s+6$, subject to absorbing boundary conditions at 
the target set~$\mathcal{H}$ and the terminal condition $E_N(s)=0$ for $s > N$.  
Since the recursion depends only on higher states, it can be solved uniquely by backward substitution.

\emph{Monotonicity in $N$.}
By Proposition~\ref{prop:identification}, the function $F(s)=\mathbb{E}_s[\min(T,\omega_N)]$
is a solution of \eqref{eq:ENdef}; by the uniqueness just established, $E_N=F$. 
Since $\omega_N$ is nondecreasing in $N$, we have 
$\min(T,\omega_{N'})\le \min(T,\omega_N)\le T$ almost surely (a.s.). 
Taking expectations gives $E_{N'}(s)\le E_N(s)\le E(s)$.
\end{proof}

For ease of reference, we record the following consequence of Propositions~\ref{prop:identification} and~\ref{prop:mono}.

\begin{corollary}[Probabilistic representation of $E_N$]\label{cor:FequalsEN}
Fix \(N\in\mathbb N\). The truncated recursion \eqref{eq:ENdef} has a unique solution \(E_N\colon \mathbb{Z}_{\ge0}\to\mathbb{R}\),
\[
E_N(s)=\mathbb{E}_s\!\big[\min(T,\omega_N)\big].
\]
Consequently, for all \(s\ge0\),
\[
E(s)-E_N(s)
=\mathbb{E}_s\!\big[(T-\omega_N)\,\mathbf{1}_{\{\omega_N<T\}}\big].
\]
\end{corollary}

\medskip
The next step is to verify that the truncated expectations $E_N(s)$ 
indeed converge to the true values $E(s)$ as the cutoff $N$ grows.
To this end, recall that the Monotone Convergence Theorem states that if 
$(Y_N)_{N\ge1}$ is a sequence of nonnegative measurable random variables 
such that $Y_N\uparrow Y$ a.s., then \cite{Billingsley}
\[
\mathbb{E}[Y_N]\;\uparrow\;\mathbb{E}[Y]
\qquad (\text{possibly }+\infty).
\]
In our setting, $Y_N=\min(T,\omega_N)$ is nonnegative and nondecreasing in $N$, 
with $Y_N\uparrow T.$ 
Thus, together with Corollary~\ref{cor:FequalsEN}, 
we immediately obtain the following corollary.

\begin{corollary}[Monotone convergence]\label{cor:monotone-conv}
For each $s\ge0$,
\[
E_N(s)=\mathbb{E}_s[\min(T,\omega_N)]
\;\uparrow\;
\mathbb{E}_s[T]=E(s)\in[0,\infty].
\]
\end{corollary}


The monotone convergence in Corollary~\ref{cor:monotone-conv} ensures that the truncated expectations $E_N(s)$ converge to the true value $E(s)$ as the cutoff~$N$ increases.
However, it does not quantify the magnitude of the truncation error for finite~$N$. The next theorem provides a rigorous quantitative bound on the truncation error in approximating~$E(s)$.

\begin{theorem}[Error bound for the truncated expectation]
\label{thm:truncation_error_strict}
Assume there exist finite constants $L_N < U_N$ (depending on $N$ and the geometry of~$\mathcal{H}$) such that 
for all $s \le N$ with $\mathbb{P}_s(A_N) > 0$,
\begin{equation}\label{eq:assump_strict}
L_N < \mathbb{E}_s\!\big[T - \omega_N \,\big|\, A_N\big] < U_N.
\end{equation}
Then,
\begin{equation}\label{eq:two_sided_strict}
0 < E(s) - \bigl(E_N(s) + L_N\,\mathbb{P}_s(A_N)\bigr)
< (U_N - L_N)\,\mathbb{P}_s(A_N).
\end{equation}
\end{theorem}

\medskip

The constants $L_N$ and $U_N$ in Theorem~\ref{thm:truncation_error_strict}
serve as lower and upper bounds on the expected residual time beyond
the truncation level~$N$.  
The probability $\mathbb{P}_s(A_N)$ corresponds to the event that the partial-sum
process exceeds~$N$ before reaching the target set~$\mathcal{H}$.  
Typically, $\mathbb{P}_s(A_N)$ decreases exponentially with the number of elements in~$\mathcal{H}$ that are $\le N$ (see Section~\ref{sec:squares}), so that the correction term $L_N\mathbb{P}_s(A_N)$ becomes negligible compared with the main term~$E_N(s)$.
Hence, in such cases, $E_N(s)$ captures the dominant contribution from trajectories that
terminate within~$[0,N]$.  
Consequently, Theorem~\ref{thm:truncation_error_strict} quantifies how closely the adjusted approximation
$E_N(s) + L_N\,\mathbb{P}_s(A_N)$ reproduces the true expectation~$E(s)$.

\begin{proof}[Proof of Theorem \ref{thm:truncation_error_strict}]
Fix $s\le N$ with $\mathbb P_s(A_N)>0$, so that the conditional expectation 
$\mathbb E_s[T-\omega_N\mid A_N]$ is well defined. 
By \eqref{eq:assump_strict} and the identity
\[
\mathbb E_s\!\big[(T-\omega_N)\mathbf 1_{A_N}\big]
=\mathbb P_s(A_N)\,\mathbb E_s\!\big[T-\omega_N\mid A_N\big],
\]
we obtain the strict bounds
\[
L_N\,\mathbb P_s(A_N)
\;<\;
\mathbb E_s\!\big[(T-\omega_N)\mathbf 1_{A_N}\big]
\;<\;
U_N\,\mathbb P_s(A_N).
\]
By Corollary~\ref{cor:FequalsEN},
$E(s)-E_N(s)=\mathbb E_s[(T-\omega_N)\mathbf 1_{A_N}]$,
and substituting into the above inequality yields
\[
L_N\,\mathbb P_s(A_N)\;<\;E(s)-E_N(s)\;<\;U_N\,\mathbb P_s(A_N).
\]
Subtracting $L_N\,\mathbb P_s(A_N)$ throughout gives
\[
0 \;<\; E(s) - \bigl(E_N(s)+L_N\,\mathbb P_s(A_N)\bigr)
\;<\; (U_N-L_N)\,\mathbb P_s(A_N),
\]
which is \eqref{eq:two_sided_strict}.
\end{proof}

\medskip


To compute the overshoot probability $\mathbb{P}_s(A_N)$, we employ a backward recursion (analogous to~\eqref{eq:ENdef}) derived via first-step analysis:
\begin{equation}\label{eq:ANrec}
\mathbb{P}_s(A_N)=
\begin{cases}
0, & s \in \mathcal H,\ s \le N,\\[3pt]
\dfrac{1}{6}\displaystyle\sum_{i=1}^6 \mathbb{P}_{s+i}(A_N), & 0 \le s \le N,\ s \notin \mathcal H,\\[6pt]
1, & s > N.
\end{cases}
\end{equation}
This relation expresses $\mathbb{P}_s(A_N)$ as the average of its six forward neighbors, with the boundary condition $\mathbb{P}_s(A_N)=1$ beyond the cutoff~$N$.
For numerical evaluation, the recursion is solved backward for
$s = N, N-1, \dots, 0$.

\medskip

\begin{remark}[Degenerate case]
If $\mathbb{P}_s(A_N)=0$, then $\mathbf{1}_{A_N}=0$ almost surely, and hence 
$E(s)-E_N(s)=0$.
In this situation, the approximation $E_N(s)$ coincides exactly with the true expectation~$E(s)$.
\end{remark}

\medskip
In Section~\ref{sec:squares}, we return to this truncated framework and specialize it to the concrete case $\mathcal{H} = \{n^2 : n \in \mathbb{N}\}$.

\section{Random Rolls and Hit/Skip Behavior}
\label{sec:randomrolls}

In this section we analyze the intrinsic random-roll dynamics of the process,
deriving universal constants for the hit and skip probabilities that later
govern the overshoot bounds $L_N$ and $U_N$ in the truncation analysis.

\medskip
We begin with a simpler random–walk model of repeated fair die rolls,
which admits explicit computation of the hit–skip probabilities
and clarifies their asymptotic behavior.


Let $(X_i)_{i\ge1}$ be i.i.d.\ random variables uniformly distributed on $\{1,\dots,6\}$, 
and define the cumulative sums
\[
W_t := X_1 + X_2 + \cdots + X_t, \qquad t=1,2,\dots .
\]
Here, $(W_t)$ represents the cumulative total after $t$ rolls.  (Note that $W_t$ starts from 0, whereas elsewhere we consider $S_t$ starting from a general $s\ge0$.)
 We now ask a basic question about the probability 
that this random walk ever visits a given integer value.

\medskip
\noindent\textbf{Question.}
Suppose we roll a fair die indefinitely. 
What is the probability that a given integer $n$ is ever attained?
Furthermore, for which $n$ is this probability maximized?

\medskip
\noindent\textbf{Answer.}
Let $p_n$ denote the probability of ever hitting $n$, that is, that $W_t = n$ for some $t$.
It turns out that the sequence $(p_n)$ \textcolor{black}{satisfies the following linear recurrence}
with constant coefficients (a so-called $C$-finite recurrence; see, e.g.,
\cite{Zeilberger2013,KauersPaule2011,KrityakierneThanatipanonda2024})\textcolor{black}{,}

\[
p_n=\frac{1}{6}\bigl(p_{n-1}+p_{n-2}+p_{n-3}+p_{n-4}+p_{n-5}+p_{n-6}\bigr),
\qquad
p_0=1,\;\; p_n=0\ \text{for }n<0.
\]
\textcolor{black}{Solving  the recurrence for $p_n$ gives}
\[
p_n=\frac{2}{7}+\frac{1}{7}u^{\,n}+\frac{1}{7}v_+^{\,n}+\frac{1}{7}v_{-}^{\,n}
+\frac{1}{7}w_+^{\,n}+\frac{1}{7}w_{-}^{\,n},
\]
where
\[
\begin{aligned}
u &= -0.6703320476, \\[4pt]
v_{\pm} &= -0.3756951992 \pm 0.5701751610\,i, 
&\quad |v_{\pm}| &= 0.6828225223, \\[4pt]
w_{\pm} &= \;\;0.2941945564 \pm 0.6683670974\,i,
&\quad |w_{\pm}| &= 0.7302499667.
\end{aligned}
\]

From this explicit form one checks that $p_n$ is maximized at $n=6$.
The first values of $p_n$ for $n\ge 1$ are
\[
\frac{1}{6},\;\frac{7}{36},\;\frac{49}{216},\;\frac{343}{1296},\;
\frac{2401}{7776},\;\frac{16807}{46656},\;\frac{70993}{279936},\;
\frac{450295}{1679616},\;\dots
\]

Since all \textcolor{black}{the} characteristic roots other than $1$ have modulus $<1$, the terms
$u^{n},v_+^{n},v_{-}^{n},w_+^{n},w_{-}^{n}$ decay exponentially, and therefore
\[
\lim_{n\to\infty} p_n=\frac{2}{7}\approx 0.28571428571.
\]
Intuitively, the partial sums advance by steps given by independent die rolls, uniformly distributed on $\{1,\dots,6\}$, \textcolor{black}{with a mean of} $3.5$.
As each integer can be visited only once, the probability $p_n$ of ever hitting $n$ approaches
$1/\mathbb{E}[X_1]=2/7$, \textcolor{black}{where $X_1$ denotes one such die roll.}
Thus, while the walk diverges to infinity, it visits only a fraction of integers—about $28.6\%$ of large $n$ are ever hit. Figure~\ref{fig:p_n} illustrates how $p_n$ evolves for $n$ from $1$ to $100$.

\begin{figure}
\begin{center}
\includegraphics[scale=0.6]{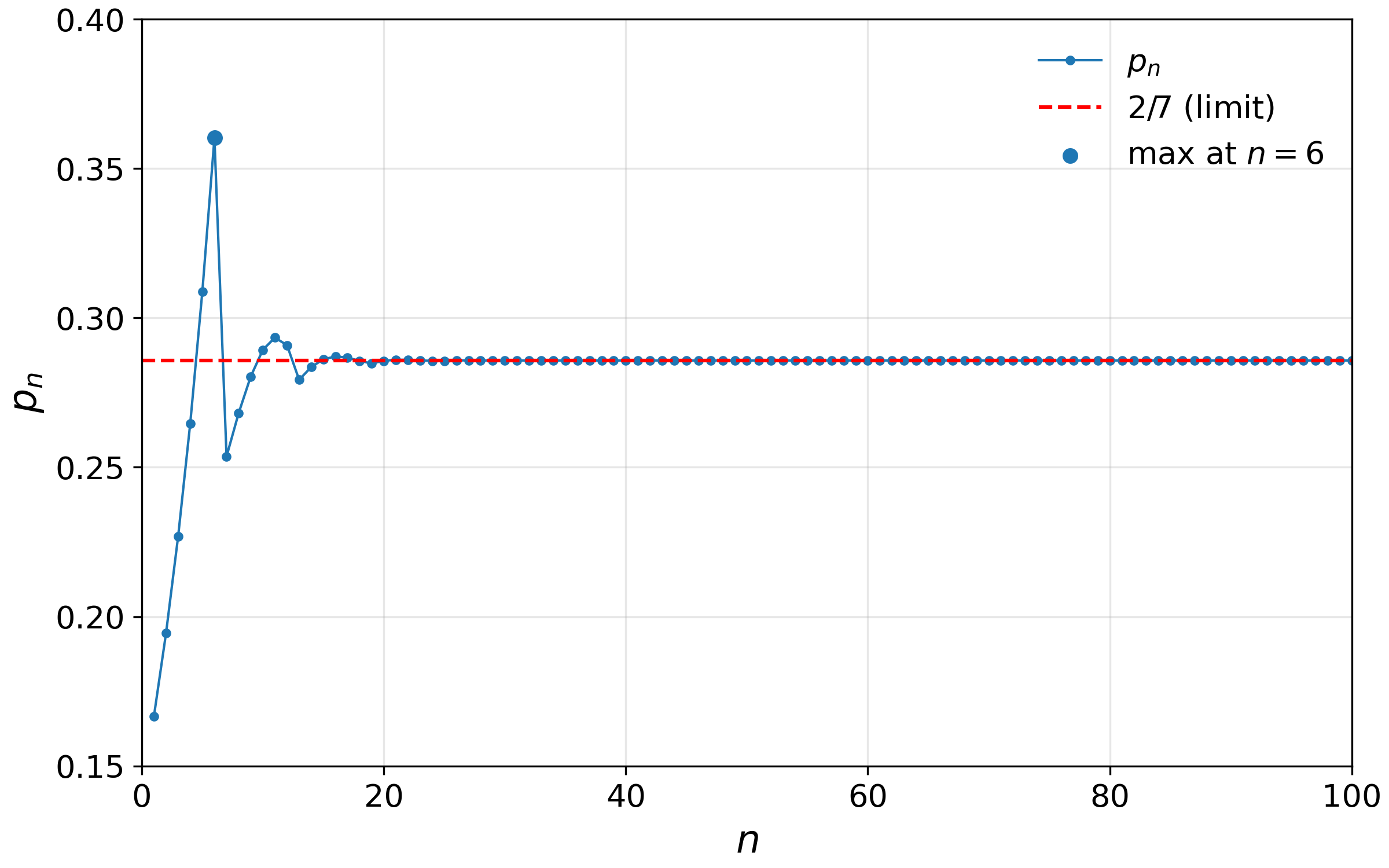}
\end{center}
\caption{Probability $p_n$ that the partial sums of fair die rolls ever hit \textcolor{black}{the integer} $n$, for $1\le n\le100$.}
\label{fig:p_n}
\end{figure}

\begin{lemma}
\label{lemma:bound_p}

Let $p_n$ denote the probability of ever hitting $n$, and let $q_n = 1 - p_n$ denote the probability of skipping (or missing) $n$.  
Then
\[
\frac{2}{7} - \varepsilon_n \le p_n \le \frac{2}{7} + \varepsilon_n,
\qquad
\frac{5}{7} - \varepsilon_n \le q_n \le \frac{5}{7} + \varepsilon_n,
\]
where $\varepsilon_n := \tfrac{5}{7}\,|w_+|^{\,n}$ and $|w_+| = 0.7302499667.$
\end{lemma}

\begin{proof}
From the explicit formula,
\[
p_n-\frac{2}{7}
= \frac{1}{7}\bigl(u^n+v_+^n+v_{-}^n+w_+^n+w_{-}^n\bigr).
\]
Hence, by the triangle inequality and since $|u|,|v_+|,|v_{-}|\le |w_+|=|w_{-}|$,
\[
\Bigl|p_n-\frac{2}{7}\Bigr|
\le \frac{1}{7}\bigl(|u|^n+|v_+|^n+|v_{-}|^n+|w_+|^n+|w_{-}|^n\bigr)
\le \frac{5}{7}\,|w_+|^n \;=\; \varepsilon_n.
\]
Because $q_n=1-p_n$, we also have
$\bigl|q_n-\tfrac{5}{7}\bigr|=\bigl|p_n-\tfrac{2}{7}\bigr|\le \varepsilon_n$,
which yields the stated bounds.
\end{proof}

Note that Alon et al.~\cite{Z2} obtained a similar asymptotic statement (their Lemma~2.1), proving that the probability that the cumulative sum of fair die rolls hits a large integer~$x$ (or~$n$), denoted~$p(x)$ (or~$p_n$), converges exponentially to~$2/7$ as~$x$ increases.  
The key contrast with Lemma~1 here lies in the explicitness of the exponential decay rate.  
Lemma~2.1 in~\cite{Z2} establishes the general inequality
\[
|p(x)-\tfrac{2}{7}| \le C(1-\mu)^x, \qquad 0<\mu<1,
\]
for some unspecified positive constants~$C$ and~$\mu$.

\section{The Perfect-Square Case}
\label{sec:squares}

We now specialize the truncation framework of Section~\ref{sec:truncated} and the
random–roll analysis of Section~\ref{sec:randomrolls} to the perfect-square target set
$\mathcal{H} = \{n^2 : n \in \mathbb{N}\}$.  
The exclusion of~$0$ is natural, as the goal is to count the number of rolls required
for the cumulative sum to reach a perfect square, which presupposes that the process
starting from zero is not immediately absorbed.

In what follows, the notation of Definition~\ref{def:stopping-times} is used with this choice of~$\mathcal{H}$.  
The intervals between successive squares form distinct bands whose skip–hit behavior
is nearly geometric: using \(p_n = \tfrac{2}{7} + O(\varepsilon_n)\) with exponentially
small \(\varepsilon_n\) (Lemma~\ref{lemma:bound_p}), we obtain bounds for the first
square reached after the cutoff.  
These yield closed-form overshoot constants \(L_N\) and \(U_N\) satisfying
the conditions of Theorem~\ref{thm:truncation_error_strict}.  
The next proposition records the corresponding band-wise probabilities.

\begin{proposition}[Geometric bounds for the first perfect square reached after $N$]
\label{prop:square_cutoff}
Let $(S_t)$ be the cumulative-sum process defined in~\eqref{eq:S}, and fix a perfect square cutoff $N = K^2$ with $K \in \mathbb{N}$ and $K \ge 4$.  
Consider the target set $\mathcal{H} = \{n^2 : n \in \mathbb{N}\}$.

For each integer $j \ge 0$, define
\[
h_j := (K + 1 + j)^2,
\]
the $(j+1)^{\mathrm{st}}$ perfect square strictly greater than~$N$.  
Then, for every initial state $s \le N$ with $\mathbb{P}_s(A_N) > 0$, the probability that the
first perfect square reached after $N$ equals $h_j$ satisfies
\[
\biggl(\frac{5}{7} - \varepsilon_N\biggr)^{\!j}
\biggl(\frac{2}{7} - \varepsilon_N\biggr)
\;\le\;
\mathbb{P}_s(S_T = h_j \mid A_N)
\;\le\;
\biggl(\frac{5}{7} + \varepsilon_N\biggr)^{\!j}
\biggl(\frac{2}{7} + \varepsilon_N\biggr),
\]
where
\[
\varepsilon_N = \frac{5}{7}\,|w_+|^{\,2K-4},
\qquad
|w_+| = 0.7302499667.
\]
\end{proposition}

\begin{proof}
For each integer \( m > K \), define the entrance time to the band \(\bigl((m-1)^2,\,m^2\bigr]\) by
\[
\tau_m := \inf\{\,t \ge 0 : S_t > (m-1)^2\,\},
\]
so that \( S_{\tau_m} \in \bigl((m-1)^2,\,m^2\bigr] \).
Conditional on the overshoot event \(A_N=\{\omega_N<T\}\), the process first enters
the band \(\bigl(K^2,(K+1)^2\bigr]\) at time \(\tau_{K+1}=\omega_N\).

When this occurs, the position satisfies
\[
S_{\tau_{K+1}}\in\{K^2+1,K^2+2,\dots,K^2+5\},
\]
since each step lies in \(\{1,\dots,6\}\)
\textcolor{black}{and $S_{\tau_{K+1}-1}\neq K^2$ under the overshoot event.}

Define the residual distance to the next perfect square by
\[
n_m := m^2 - S_{\tau_m}.
\]
Then, \(n_{K+1}\ge (K+1)^2-(K^2+5)=2K-4\).

For later bands \(((K+j)^2,(K+1+j)^2]\) with \(j\ge 0\), a similar argument gives
\[
n_{K+1+j}\;\ge\;(K+1+j)^2 - \bigl((K+j)^2+5\bigr)\;=\;2(K+j)-4\;\ge\;2K-4.
\]
Hence,  \(n_{\min}:=2K-4\) is a uniform lower bound for all such gaps.

By Lemma~\ref{lemma:bound_p},
\[
\Bigl|p_n-\tfrac{2}{7}\Bigr|\le\varepsilon_n,
\qquad
\varepsilon_n=\tfrac{5}{7}|w_+|^{\,n},
\]
so the hit probability in any band lies in \(\tfrac{2}{7}\pm\varepsilon_n\)
and the skip probability in \(\tfrac{5}{7}\pm\varepsilon_n\).
Since \(\varepsilon_n\) decreases with \(n\) and \(n\ge n_{\min}\) in all bands,
we may take the uniform bound \(\varepsilon_N:=\varepsilon_{n_{\min}}
=\tfrac{5}{7}|w_+|^{\,2K-4}\).

To have the first square after \(N\) equal \(h_j=(K+1+j)^2\),
the walk must skip the first \(j\) squares \((K+1)^2,\dots,(K+j)^2\)
and then hit \((K+1+j)^2\).
By the strong Markov property applied at the successive entrance times
\textcolor{black}{$\tau_{K+1}, \dots, \tau_{K+1+j}$}, the law in each band depends only on the current entrance state, 
and by iterated conditioning, the corresponding probabilities multiply.
Replacing each skip and hit factor by its respective uniform bound
\(\tfrac{5}{7}\pm\varepsilon_N\) or \(\tfrac{2}{7}\pm\varepsilon_N\)
yields the stated inequalities.
\end{proof}

The next proposition explicitly constructs the bounds $L_N$ and $U_N$ required for applying Theorem~\ref{thm:truncation_error_strict} in the perfect-square case.

\begin{proposition}[Overshoot bound]
\label{prop:overshoot}
Assume the setup and notation of Proposition~\ref{prop:square_cutoff}, and let $K = \sqrt{N} \in \mathbb{N}$.  
Define
\begin{align*}
L_N &:= \frac{1}{6}\sum_{j=0}^{\infty}
\Bigl((K+1+j)^2 - K^2 - 5\Bigr)\,
\Bigl(\tfrac{5}{7}-\varepsilon_N\Bigr)^j
\Bigl(\tfrac{2}{7}-\varepsilon_N\Bigr), \\[4pt]
U_N &:= \sum_{j=0}^{\infty}
\Bigl((K+1+j)^2 - K^2 - 1\Bigr)\,
\Bigl(\tfrac{5}{7}+\varepsilon_N\Bigr)^j
\Bigl(\tfrac{2}{7}+\varepsilon_N\Bigr).
\end{align*}
Then, for all $K \ge 4$ and every $s \in \{0,1,\dots,N\}$,
\begin{equation}\label{eq:cond-overshoot-PDF}
L_N \;<\; \mathbb{E}_s\!\bigl[T-\omega_N \,\big|\, A_N\bigr]
 \;<\; U_N.
\end{equation}
\end{proposition}


\begin{proof}

Fix $s \in \{0,1,\dots,N\}$ with $N = K^2$.  
We now work conditional on the overshoot event 
$A_N = \{\omega_N < T\}$, that is, on those sample paths for which the walk 
crosses the cutoff $N$ before hitting any perfect square.

\medskip

\noindent\textit{Step 1: Deterministic time bounds within a band (strict in expectation).}
For $j \ge 0$, set $h_j = (K + 1 + j)^2$ and define
\[
B_j := \{\, S_T = h_j \,\},
\]
the event that the first perfect square reached after crossing the cutoff $N = K^2$ is exactly $h_j$.

On $A_N$, the first entrance into $\bigl(K^2,(K+1)^2\bigr]$ satisfies
$S_{\omega_N}\in\{K^2+1,\dots,K^2+5\}$, hence the gap to $h_j$ obeys
\[
h_j-S_{\omega_N}\in\bigl[(K+1+j)^2-K^2-5,\ (K+1+j)^2-K^2-1\bigr].
\]
Since each step lies in $\{1,\dots,6\}$, the hitting time from $S_{\omega_N}$ to $h_j$ satisfies
\[
\frac{h_j-S_{\omega_N}}{6}\ \le\ T-\omega_N\ \le\ h_j-S_{\omega_N}
\qquad\text{on }B_j.
\]
Taking conditional expectations yields the \emph{strict} bounds
\begin{equation}\label{eq:band-bounds}
\frac{(K+1+j)^2-K^2-5}{6}\ <\ \mathbb{E}_s\!\bigl[T-\omega_N\,\big|\,B_j,A_N\bigr]
\ <\ (K+1+j)^2-K^2-1 ,
\end{equation}
since equality in either direction would require all rolls after $\omega_N$ to be identical
($6$ for the lower, $1$ for the upper bound), an event of probability strictly less than~one.

\medskip
\noindent\textit{Step 2: Use Proposition~\ref{prop:square_cutoff} for band weights.}
By Proposition~\ref{prop:square_cutoff}, for all $j\ge0$,
\[
\bigl(\tfrac{5}{7}-\varepsilon_N\bigr)^j\bigl(\tfrac{2}{7}-\varepsilon_N\bigr)
\ \le\ \mathbb{P}_s\!\bigl(B_j\,\big|\,A_N\bigr)
\ \le\
\bigl(\tfrac{5}{7}+\varepsilon_N\bigr)^j\bigl(\tfrac{2}{7}+\varepsilon_N\bigr),
\qquad
\varepsilon_N=\tfrac{5}{7}|w_+|^{\,2K-4}.
\]

\medskip
\noindent\textit{Step 3: Combine using band decomposition.}
Decomposing on $\{B_j\}_{j\ge0}$,
\[
\mathbb{E}_s\!\bigl[T-\omega_N\,\big|\,A_N\bigr]
=\sum_{j=0}^\infty \mathbb{E}_s\!\bigl[T-\omega_N\,\big|\,B_j,A_N\bigr]\,
\mathbb{P}_s\!\bigl(B_j\,\big|\,A_N\bigr).
\]
Combining \eqref{eq:band-bounds} with the result from \textit{Step 2}
yields exactly the stated lower and upper series $L_N$ and $U_N$,
proving \eqref{eq:cond-overshoot-PDF}.

\medskip
\noindent\textit{Convergence note.}
Since \(\tfrac{5}{7}+\varepsilon_N<1\) whenever \(K\ge4\), both series converge absolutely for all \(K\ge4\), which completes the proof.

\end{proof}

Proposition~\ref{prop:overshoot} establishes the bounds
\textcolor{black}{in Equation}~\eqref{eq:cond-overshoot-PDF} required by the hypothesis 
of Theorem~\ref{thm:truncation_error_strict}.  
As a direct consequence of these two results, we obtain the following corollary.

\begin{corollary}[Specialization of Theorem~\ref{thm:truncation_error_strict} to the perfect-square case]
\label{cor:trunc-error}
Let $\mathcal H = \{n^2 : n \in \mathbb N\}$ and $N = K^2$ with $K \ge 4$.
Then, for every initial state $s \in \{0,1,\dots,N\}$,
\[
0 < E(s) - \bigl(E_N(s) + L_N\,\mathbb P_s(A_N)\bigr)
< (U_N - L_N)\,\mathbb P_s(A_N),
\]
where $L_N$ and $U_N$ are defined in Proposition~\ref{prop:overshoot}.
\end{corollary}

\subsection*{Computational Implementation}

The geometric-series bounds in Proposition~\ref{prop:overshoot} were implemented in closed form using the analytic identities
$\sum_{j\ge0} r^j$, $\sum_{j\ge0} j r^j$, and $\sum_{j\ge0} j^2 r^j$
(see Appendix~\ref{appendix:LNUN-closedform} for the detailed derivation).
The quantities $E_N(0)$ and $\mathbb{P}_0(A_N)$ were evaluated from the recursions
\eqref{eq:ENdef} and~\eqref{eq:ANrec}.
Both computations employed a rolling-window update that retains only the six most recent values in the recurrence, achieving $O(1)$ storage while remaining equivalent to $O(N)$ backward computation.
This implementation reduced the total runtime by more than a factor of two relative to a naïve array approach, with no loss of numerical accuracy.

We now apply Corollary~\ref{cor:trunc-error}  with initial state $s=0$, evaluating the bound using the 
explicit expressions for $L_N$, $U_N$, and $\mathbb P_0(A_N)$.
All computations were performed on an Apple MacBook Pro 
equipped with an M1~Pro chip (10-core CPU with 8 performance and 2 efficiency cores, 
16~GB RAM).  
A custom Python implementation was developed using the 
\texttt{mpmath} library for arbitrary-precision arithmetic, 
with the working precision set to \verb|mp.dps|~$=1200$.  
The complete run for $N = 7000^2$ finished in 653.18~seconds and produced the 
value of~$E(0)$ reported below. \textcolor{black}{A Python notebook implementation reproducing the results of Theorem~\ref{thm:expected-square} is available at \texttt{\url{https://thotsaporn.com/HitSquare/IO1.html}}.}

\medskip

\begin{theorem}[Expected hitting time to the perfect squares]
\label{thm:expected-square}
The expected number of rolls required for the cumulative die–roll sum
to reach a perfect square is (1,018 total digits):
\begin{NumWrap}
E(0) = 7.079764237551105103895553056908184894681711444263208805908873101517293030636657289150619440215929586140643853058236617839038805437427037161983225198843501869295681377649823444071523388800882074553106810227935191220149739931296954376558933192195369394958351011153114111799919088138505138599357264273458295534653653705548720477130373704649449607046275208840820791615363183593786984085594202052884475208247842900518291457801426255494832590823030504774813684129030383618661091994729346316899126658258686744721776623692164339998786486070630256359172293213930131126660661305353727091275283033821909596371164463346351343143276559533679088929433795409592073773399518296416540470294895323623629224992997477760853053903893989731353287187931136704436094234746646663976703439467712341171718619085317407308552387899094073533086262057687143557474061763457398136241182138420814929896478348546126586312504509048656042736191857511098779116613148179664850379978987656091676501299944201086790790737071543289307884197271970205089906775387...
\end{NumWrap}
This value of $E(0)$ is correct to at least \textbf{1,017~decimal places}.
\end{theorem}


\begin{proof}
Let $N=7000^2$ and consider the truncated recursion~\eqref{eq:ENdef}. 
By Corollary~\ref{cor:trunc-error}, for $s=0$ we have
\[
0
< E(0) - \bigl(E_N(0) + L_N\,\mathbb{P}_0(A_N)\bigr)
< (U_N - L_N)\,\mathbb{P}_0(A_N).
\]
The computed quantities are
\[
\begin{aligned}
\mathbb{P}_0(A_N) &= 1.508850331472307815412722898448210123557\times 10^{-1023},\\[2mm]
L_N &= 8169.\overline{3}, \qquad
U_N = 49020,
\end{aligned}
\]
which yield the rigorous upper bound on the remainder
\[
(U_N - L_N)\,\mathbb{P}_0(A_N)
= 6.163754194086475579815333888354168235404\times 10^{-1019}.
\]
Hence,  the approximation $E_N(0) + L_N\,\mathbb{P}_0(A_N)$ agrees with $E(0)$ to at least 1{,}017~decimal places, as claimed.
\end{proof}


\section{Beyond Perfect Squares}
\label{sec:future}

The preceding analysis focused on the expected hitting time of the cumulative-sum process to the target set $\mathcal{H} = \{n^2 : n \in \mathbb{N}\}$.
Because all results were established for an arbitrary initial state~$S_0 = s$, the same method yields high-precision values of~$E_s[T]$ for any~$s \ne 0$.
Consequently, questions of the type posed in~\cite{Z1}, such as
\textit{“What is the expected number of rolls needed to hit a perfect square, starting from~$s$?”},
can be answered directly by our framework.
\textcolor{black}{The same framework also extends directly to a fair $M$-sided die with arbitrary
$M\ge2$, with increments uniformly distributed on $\{1,2,\dots,M\}$.
Thus, the method applies beyond the classical six-sided case; see, for
example, Martinez and Zeilberger \cite{Z1} and Chern \cite{Chern2024}.}

In general, a comprehensive analysis of this decay behavior would clarify how the decay rate of~$\mathbb{P}_s(A_N)$ governs the attainable precision of the method, independently of the sharpness of the auxiliary bounds~$L_N$ and~$U_N$.
As seen in the perfect-square case~$\mathcal{H}=\{n^2:n\in\mathbb{N}\}$, $\mathbb{P}_0(A_N)$ is exceedingly small, about $10^{-1023}$, making the error bound $(U_N - L_N)\mathbb{P}_0(A_N)$ effectively negligible.
Thus, once the uniform overshoot condition~\eqref{eq:assump_strict} holds, the overall accuracy of the truncated recurrence depends almost entirely on~$\mathbb{P}_s(A_N)$, which varies according to the rarity of the target set~$\mathcal{H}$.

A particularly appealing extension is to take~$\mathcal{H}$ as the set of prime numbers.
Because primes are asymptotically denser than perfect squares, with
$\pi(x)\sim x/\log x$ compared with~$\sqrt{x}$, the probability term
$\mathbb{P}_s(A_N)$ in Theorem~\ref{thm:truncation_error_strict} is expected to decay more rapidly
under the same cutoff~$N$. As a result, the remainder term
$(U_N-L_N)\mathbb{P}_s(A_N)$ is likely to be smaller in magnitude, leading to improved numerical accuracy.
\textcolor{black}{Unlike the perfect-square case, however, the quantities $L_N$ and $U_N$
from Proposition~\ref{prop:overshoot} would need to be derived by different methods, since the
spacing of primes is irregular. In particular, asymptotic tools such as the
Prime Number Theorem may be useful for estimating the number of primes beyond
the cutoff and for constructing corresponding overshoot bounds.}
This highlights the generality of the proposed framework and demonstrates its
applicability to other target sets~$\mathcal{H}$ beyond the case of perfect
squares. \textcolor{black}{Exploring this direction would be an interesting topic for future work.}

\section*{Appendix}
\appendix
\section{Closed-form computation of $L_N$ and $U_N$}
\label{appendix:LNUN-closedform}

The bounds $L_N$ and $U_N$ from Proposition~\ref{prop:overshoot} are defined by the convergent series
\[
\begin{aligned}
L_N
&=\frac{1}{6}\sum_{j=0}^{\infty}
\bigl((K+1+j)^2-K^2-5\bigr)
\Bigl(\tfrac{5}{7}-\varepsilon_N\Bigr)^{j}
\Bigl(\tfrac{2}{7}-\varepsilon_N\Bigr),\\[4pt]
U_N
&=\sum_{j=0}^{\infty}
\bigl((K+1+j)^2-K^2-1\bigr)
\Bigl(\tfrac{5}{7}+\varepsilon_N\Bigr)^{j}
\Bigl(\tfrac{2}{7}+\varepsilon_N\Bigr),
\end{aligned}
\]
where $\varepsilon_N=\tfrac{5}{7}|w_+|^{\,2K-4}$, $K=\sqrt{N}$, and $|w_+|=0.7302499667$.

To evaluate $L_N$ and $U_N$ in closed form, we apply the standard geometric-series identities
\[
\sum_{j\ge0} r^j = \frac{1}{1-r}, \qquad 
\sum_{j\ge0} j r^j = \frac{r}{(1-r)^2}, \qquad
\sum_{j\ge0} j^2 r^j = \frac{r(1+r)}{(1-r)^3}, \qquad |r|<1.
\]

For $d\in\{5,1\}$ and parameters $r,t$ with $|r|<1$, define
\[
\begin{aligned}
\Sigma(d;r,t) \;:=\; \sum_{j\ge0}\!\bigl[(K+1+j)^2-K^2-d\bigr]\,r^{\,j}\,t
&= \sum_{j\ge0} \!\bigl[(2K+1-d) + 2(K+1)j + j^2\bigr] r^j\, t \\[2pt]
&= t\!\left[
\tfrac{2K+1-d}{1-r}
+ \tfrac{2(K+1)r}{(1-r)^2}
+ \tfrac{r(1+r)}{(1-r)^3}
\right].
\end{aligned}
\]

Introduce
\[
r_\pm=\tfrac{5}{7}\pm\varepsilon_N, \qquad t_\pm=\tfrac{2}{7}\pm\varepsilon_N.
\]

Then,
\[
L_N= \tfrac{1}{6}\,\Sigma(5;\,r_-,t_-)
\qquad\text{and}\qquad
U_N=\Sigma(1;\,r_+,t_+).
\]

When $\varepsilon_N=0$ (so $r_\pm=\tfrac{5}{7}$ and $t_\pm=\tfrac{2}{7}$), these simplify to the explicit linear forms
\[
L_N=\tfrac{7}{6}K+\tfrac{8}{3},
\qquad
U_N=7K+20.
\]

In our implementation in Theorem~\ref{thm:expected-square}, however, we retain
\[
r_{\pm}=\tfrac{5}{7}\pm\varepsilon_N
\quad\text{and}\quad
t_{\pm}=\tfrac{2}{7}\pm\varepsilon_N
\]
which yields nearly identical values of $L_N$ and $U_N$, since $\varepsilon_N$ is very small.



\end{document}